\documentclass[12pt]{amsart}
\usepackage{amssymb}
\usepackage[cmtip,all]{xy}
\usepackage{verbatim}
\usepackage{upref}
\usepackage[usenames,dvipsnames]{xcolor}
\usepackage{url}
\usepackage[normalem]{ulem}

\pdfoutput=1

\newtheorem{thm}{Theorem}[section]
\newtheorem*{thm*}{Theorem}
\newtheorem{lem}[thm]{Lemma}
\newtheorem{cor}[thm]{Corollary}
\newtheorem{prop}[thm]{Proposition}

\theoremstyle{definition} 
\newtheorem{defn}[thm]{Definition}

\theoremstyle{remark} 
\newtheorem{rem}[thm]{Remark}
\newtheorem{q}[thm]{Question}

\newcommand{\secref}[1]{Section~\textup{\ref{#1}}}

\newcommand{\thmref}[1]{Theorem~\textup{\ref{#1}}}
\newcommand{\corref}[1]{Corollary~\textup{\ref{#1}}}
\newcommand{\lemref}[1]{Lemma~\textup{\ref{#1}}}
\newcommand{\propref}[1]{Proposition~\textup{\ref{#1}}}

\numberwithin{equation}{section}

\newcommand{\midtext}[1]{\quad\text{#1}\quad}
\newcommand{\righttext}[1]{\quad\text{#1 }}
\renewcommand{\and}{\midtext{and}}

\newcommand{\C}{\mathbb C}
\newcommand{\T}{\mathbb T}

\newcommand{\CC}{\mathcal C}

\newcommand{\KK}{\mathcal K}
\newcommand{\II}{\mathcal I}

\newcommand{\UU}{\mathcal U}

\renewcommand{\epsilon}{\varepsilon}

\DeclareMathOperator{\aut}{Aut}

\DeclareMathOperator{\ad}{Ad}

\DeclareMathOperator*{\spn}{span}
\DeclareMathOperator*{\clspn}{\overline{\spn}}

\newcommand{\id}{\text{\textup{id}}}

\renewcommand{\>}{\rangle}
\renewcommand{\iff}{\ensuremath{\Leftrightarrow}}
\newcommand{\then}{\ensuremath{\Rightarrow}}

\newcommand{\inv}{^{-1}}

\newcommand{\variso}{\overset{\simeq}{\longrightarrow}}
\renewcommand{\bar}{\overline}
\newcommand{\what}{\widehat}

\renewcommand{\:}{\colon}

\newcounter{nogo}
\newcommand{\nogothm}{\stepcounter{nogo}No-Go Theorem~\thenogo}

\begin{document}
\title[Rigidity for $C^*$-dynamical systems]{Rigidity theory for $C^*$-dynamical systems and the ``Pedersen Rigidity Problem''}
\author[Kaliszewski]{S.~Kaliszewski}
\address{School of Mathematical and Statistical Sciences
\\Arizona State University
\\Tempe, Arizona 85287}
\email{kaliszewski@asu.edu}
\author[Omland]{Tron Omland}
\address{Department of Mathematics
\\University of Oslo
\\P.O. Box 1053 Blindern
\\NO-0316 Oslo
\\Norway}
\email{trono@math.uio.no}
\author[Quigg]{John Quigg}
\address{School of Mathematical and Statistical Sciences
\\Arizona State University
\\Tempe, Arizona 85287}
\email{quigg@asu.edu}

\date{\today}
\date{December 29, 2017}

\thanks{The second author is funded by the Research Council of Norway through FRINATEK, project no.~240913.}


\begin{abstract}
Let $G$ be a locally compact abelian group.
By modifying a theorem of Pedersen,
it follows that
actions
of $G$
on $C^*$-algebras $A$ and $B$
are outer conjugate
if and only if there is an isomorphism
of the crossed products
that is equivariant for
the dual actions
and preserves the images of $A$ and $B$ in the multiplier algebras of the crossed products.
The rigidity problem discussed in this paper
deals with the necessity of the last condition
concerning the images of $A$ and $B$.

There
is an alternative formulation of the problem:
an action of the dual group $\what G$
together with a suitably equivariant unitary homomorphism of $G$
give rise to a generalized fixed-point algebra
via Landstad's theorem,
and 
a 
problem related to the above is to produce an action of $\what G$
and two such equivariant unitary homomorphisms of $G$
that give
distinct generalized fixed-point algebras.

We present several situations where the condition
on the images of $A$ and $B$
is redundant, and where
having distinct generalized fixed-point algebras
is impossible. For example, if $G$ is discrete, this will be the case for all actions of $G$.
\end{abstract}

\maketitle

\section{Introduction}\label{intro}

When presented with a
$C^*$-dynamical system,
which we call simply an \emph{action},
one of our first impulses is to
form the crossed product $C^*$-algebra.
A fundamental question arises:
how do we recover the action from the crossed product?
The short answer is:
we 
cannot
the crossed-product $C^*$-algebra is not enough information.
So,
the next questions are:
(1) what extra data do we need to recover the action,
and
(2) in what sense can the action be recovered?
Crossed-product duality is largely devoted to these questions.

We 
consider
actions of a fixed group $G$,
and we want to recover $(A,\alpha)$.
If we are only given $A\rtimes_\alpha G$,
we cannot even recover the $C^*$-algebra $A$ up to Morita equivalence.
In other words, there exist non-Morita-equivalent $C^*$-algebras $A$ and $B$ carrying actions of a group $G$ such that $A\rtimes G\simeq B\rtimes G$.
Instances of this are surprisingly hard to find in the literature,
but for example \cite[Remark~4.3]{adic} shows how to do it with $A$ and $B$ commutative and $G$ discrete abelian.


We will 
assume throughout that $G$ is abelian.
Most of the following is true for nonabelian $G$,
but this involves coactions, which would tend to obscure the heart of the matter.
If $G$ is an abelian group,
$A$ is a $C^*$-algebra,
and $\alpha\:G\to \aut A$ is an action,
then there is an action $\what\alpha$,
called the \emph{dual action},
of the dual group $\what G$ on the crossed product $A\rtimes_\alpha G$,
and the (Takesaki-)Takai theorem says that the crossed product by the dual action
is isomorphic to $A\otimes \KK(L^2(G))$.
Moreover, the double dual action corresponds to
the tensor product of $\alpha$ with conjugation by the right regular representation $\rho$ of $G$.
Thus, Takai duality recovers 
the original action $(A,\alpha)$ of $G$ 
up to Morita equivalence from the dual action $(A\rtimes_\alpha G,\what\alpha)$ of $\what G$.

To recover more about $\alpha$, we need more information about the crossed product.
Raeburn's abstract definition of the crossed product, using universal properties,
gives a covariant homomorphism $(i_A,i_G)$ of $(A,G)$ in the multiplier algebra $M(A\rtimes_\alpha G)$
such that every other covariant homomorphism $(\pi,U)$ in $M(B)$ factors through a homomorphism $\pi\times U\:A\rtimes G\to M(B)$.
Landstad duality says that the action $(A,\alpha)$ can be recovered up to conjugacy (equivariant isomorphism) from the data $(A\rtimes_\alpha G,\what\alpha,i_G)$.
Abstractly, we are given an \emph{equivariant action} $(C,\zeta,V)$ of $\what G$,
i.e., an action $(C,\zeta)$ of $\what G$ and a strictly continuous unitary homomorphism $V\:G\to M(C)$ that is equivariant for $\zeta$ and the action $\alpha^G$ of $\what G$ on $C^*(G)$ determined by $\alpha^G_\gamma(s)=\gamma(s)s$ for $\gamma\in \what G,s\in G$
where we freely identify elements of $G$ with unitaries in $M(C^*(G))$.
The rough idea is to construct $i_A(A)$ as a \emph{generalized fixed-point algebra} $C^{\zeta,V}$ of the equivariant action.

In \cite{koqlandstad} we explored a duality intermediate between Takai and Landstad duality, that we called \emph{outer duality}, but which perhaps deserves to be christened ``Pedersen duality''.
This intermediate duality is based upon Pedersen's theorem:
two actions $\alpha$ and $\beta$ of $G$ on $A$ are exterior equivalent
if and only if there is an isomorphism $\Theta\:(A\rtimes_\alpha G,\what\alpha)\variso (A\rtimes_\beta G,\what\beta)$ such that $\Theta\circ i_A^\alpha=i_A^\beta$.
Escaping from the single $C^*$-algebra $A$, an alternative version of Pedersen's theorem (see \thmref{ped outer} below) says that
two actions $(A,\alpha)$ and $(B,\beta)$ of $G$ are outer conjugate
if and only if there is an isomorphism $\Theta\:(A\rtimes_\alpha G,\what\alpha)\variso (B\rtimes_\beta G,\what\beta)$ such that $\Theta(i_A(A))=i_B(B)$.
Thus, Pedersen shows how to recover the action $(A,\alpha)$ up to outer conjugacy if we know the dual action $(C,\zeta)$ and the generalized fixed-point algebra $C^{\zeta,V}$,
but perhaps not the equivariant homomorphism $V$ itself.

After proving the outer duality theorem, we naturally wanted to give examples exploring the boundaries of Pedersen's theorem.
More precisely, we searched for examples of two actions $(A,\alpha)$ and $(B,\beta)$ of $G$ such that the dual actions $\what\alpha$ and $\what\beta$ are conjugate,
but $\alpha$ and $\beta$ are not outer conjugate.
Equivalently, we want there to be an isomorphism between the dual actions, but not one that preserves the generalized fixed-point algebras.
Somehow surprisingly, we were not able to produce any examples of this phenomenon, and moreover, we discovered a complete absence of such examples in the literature.
This is striking, since it is tempting to conjecture that one of the first questions researchers must have asked about crossed products is,
how much information do we get from only knowing the dual actions?
It seems to us that this investigation is long overdue.

Thus we are led to what we call the \emph{Pedersen rigidity problem}:
either find explicit examples of two non-outer-conjugate actions with conjugate dual actions,
or prove that this cannot happen.
Such a theorem would (at least for abelian groups) result in a significant simplification of Pedersen's theorem, namely removing the clause about the generalized fixed-point algebras;
we suspect that counterexamples do in fact exist.
Our ultimate goal is a definitive answer to this question, but so far we have made only partial progress,
namely we have found a number of ``no-go theorems'':
the phenomenon of 
non-outer-conjugate actions with conjugate dual actions
cannot occur if
(1) $G$ is discrete,
(2) $A$ and $B$ are stable,
(3) $A$ and $B$ are commutative,
(4) $\alpha$ or $\beta$ is inner,
(5) $G$ is compact and $\alpha$ and $\beta$ are faithful and ergodic,
or
(6) $A$ is covered by invariant ideals on which the action is \emph{strongly Pedersen rigid} in the sense that the dual action has only one generalized fixed-point algebra.
In light of these no-go theorems, it is apparent that the phenomenon of multiple generalized fixed-point algebras is delicate.

We begin in \secref{prelim} by recording our conventions regarding actions, crossed products, outer conjugacy, and Pedersen's theorem(s).
In subsequent sections we prove the no-go theorems.
The final such theorem includes the case of locally unitary actions on continuous-trace $C^*$-algebras,
and we explain in a remark that this context was in fact the germ of the idea leading to the last no-go theorem.

Some of this research was done during a visit of the third author to the University of Oslo, and he thanks Erik B\'edos, Nadia Larsen, and Tron Omland for their hospitality.
Other parts of the work were done during the second author's visit to Arizona State University, and he is grateful to the analysis group at ASU for their hospitality.

We thank the referee for comments that improved the quality of this paper.

\section{Preliminaries}\label{prelim}

Throughout, $G$ will be a locally compact abelian group, $A$ a $C^*$-algebra,
and $\alpha\:G\to \aut A$ an action of $G$ on $A$.
Since $G$ will be fixed, we just say that $(A,\alpha)$ is an action. 
We adopt the conventions of \cite[Appendix~A]{enchilada} for actions and crossed products,
and here we recall the basic notation and results we will need.
If $(A,\alpha)$ is an action and $B$ is a $C^*$-algebra,
a \emph{covariant homomorphism} of $(A,\alpha)$ in $M(B)$ is a pair $(\pi,U)$,
where $\pi\:A\to M(B)$ is a nondegenerate homomorphism and $U\:G\to M(B)$ is a strictly continuous unitary homomorphism such that
\[
\pi\circ\alpha_s(a)=\ad U_s\circ\pi(a)=U_s\pi(a)U_s^*\righttext{for all}s\in G,a\in A.
\]
The \emph{crossed product}
of $(A,\alpha)$
is a triple $(A\rtimes_\alpha G,i_A,i_G)$,
where $A\rtimes_\alpha G$ is a $C^*$-algebra
and $(i_A,i_G)$ is a covariant homomorphism of $(A,\alpha)$ in $M(A\rtimes_\alpha G)$ with the \emph{universal property} that for any covariant homomorphism $(\pi,U)$ of $(A,\alpha)$ in $M(B)$ there is a unique nondegenerate homomorphism $\pi\times U\:A\rtimes_\alpha G\to M(B)$,
called the \emph{integrated form} of $(\pi,U)$,
such that
\[
(\pi\times U)\circ i_A=\pi\midtext{and}(\pi\times U)\circ i_G=U.
\]
Sometimes we write $(i_A^\alpha,i_G^\alpha)$ if ambiguity is possible,
and on the other hand we sometimes write $A\rtimes G$ if $\alpha$ is understood.
By definition, the crossed product is unique up to isomorphism in the sense that if $(B,\pi,U)$ is another crossed product then there is a unique isomorphism $\theta\:A\rtimes_\alpha G\variso B$ such that $\theta\circ i_A=\pi$ and $\theta\circ i_G=U$.

Given an action $(A,\alpha)$,
because $G$ is abelian there is a unique action $\what\alpha$ of the dual group $\what G$ on $A\rtimes_\alpha G$ such that
\[
\what\alpha_\gamma\circ i_A=i_A
\midtext{and}
\what\alpha_\gamma\circ i_G(s)=\gamma(s)i_G(s)
\righttext{for all}s\in G,\gamma\in \what G.
\]

A \emph{conjugacy}, or \emph{isomorphism}, between two actions $(A,\alpha)$ and $(B,\beta)$ is an isomorphism $\phi\:A\variso B$ that is $\alpha-\beta$ equivariant in the sense that $\phi\circ\alpha_s=\beta_s\circ\phi$ for all $s\in G$, and $\alpha$ and $\beta$ are \emph{conjugate}, or \emph{isomorphic}, if such a $\phi$ exists.
Given a conjugacy $\phi$, the pair $(i_B\circ\phi,i_G^\beta)$ is a covariant homomorphism of $(A,\alpha)$, and the integrated form $\phi\rtimes G$ is a conjugacy $(A\rtimes_\alpha G,\what\alpha)\variso (B\rtimes_\beta G,\what\beta)$.

\emph{\(Takesaki-\)Takai duality} \cite[Theorem~3.4]{takai} says that
the double dual action
\[
\bigl(A\rtimes_\alpha G\rtimes_{\what\alpha} G,\what{\what\alpha}\bigr)
\]
is conjugate to $(A\otimes\KK(L^2(G)),\id\otimes \ad\rho)$
(the part about the double dual action is not stated in \cite{takai}, but 
appears in \cite[Theorem~7.9.3]{ped}, for example).

Two actions $(A,\alpha)$ and $(B,\beta)$ of $G$ are \emph{Morita equivalent} if there exist
an $A-B$ imprimitivity bimodule $X$ and an $\alpha-\beta$ compatible action $u$ of $G$ on $X$,
i.e., for all $s\in G$ and $x,y\in X$ we have
\begin{align*}
\alpha_s\bigl({}_A\<x,y\>\bigr)&={}_A\<u_s(x),u_s(y)\>
\\
\beta_s\bigl(\<x,y\>_B\bigr)&=\<u_s(x),u_s(y)\>_B.
\end{align*}
By \cite[Remark~7.3]{tfb}, the above properties imply that $u$ is also compatible with the bimodule structure,
i.e.,
\[
u_s(ax)=\alpha_s(a)u_x(x)
\midtext{and}
u_s(xb)=u_s(x)\beta_s(b)
\]
for all $a\in A,x\in X,b\in B$.
If $\alpha$ and $\beta$ are Morita equivalent, then so are the dual actions $\what\alpha$ and $\what\beta$;
for the crossed-product $C^*$-algebras, this Morita equivalence is proved in
\cite[Section~6, Theorem]{com} and \cite[Theorem~1]{cmw},
and the existence of an $\what\alpha-\what\beta$ compatible action on the crossed-product imprimitivity bimodule follows from \cite[Proposition~3.5]{enchilada}
(see also \cite[Lemma~3.3]{KusudaDuality}).
It follows from Takai duality that the double dual action $(A\rtimes_\alpha G\rtimes_{\what\alpha} G,\what{\what\alpha})$ is Morita equivalent to $(A,\alpha)$.

An \emph{inner action} of $G$ on $A$ is one of the form $\ad u$, where $u\:G\to M(A)$ is a strictly continuous unitary homomorphism.
In this case there is a conjugacy
\[
(A\rtimes_\alpha G,\what\alpha)\simeq (A\otimes C^*(G),\id_A\otimes \alpha^G),
\]
where $\alpha^G$ is the unique action of $\what G$ on $C^*(G)$ such that $\alpha^G_\gamma(s)=\gamma(s)s$ for all $\gamma\in \what G,s\in G$.
Note that the above terminology is somewhat inconsistent within the literature. What we call an inner action is sometimes called a ``unitary action''. Moreover, we will still use the notion ``locally unitary action'' later on, e.g.\ in \corref{nogo local}.

\emph{Landstad duality} \cite[Theorem~3]{lan:dual}
(stated for abelian groups in \cite[Theorem~28]{pedersenexterior},
and in somewhat more detail in \cite[Theorem~2.2]{koqlandstad})
says that, for a $C^*$-algebra $C$,
there exist an action $(A,\alpha)$ of $G$ and an isomorphism $\theta\:A\rtimes_\alpha G\variso C$
if and only if there exist an action $\zeta$ of $\what G$ on $C$ and a $\alpha^G-\zeta$ equivariant strictly continuous unitary homomorphism $V\:G\to M(C)$.
Moreover, given such a triple $(C,\zeta,V)$, which we call an \emph{equivariant action} of $\what G$,
the action $(A,\alpha)$ and the isomorphism $\theta$ can be chosen such that $\theta$ is $\what\alpha-\zeta$ equivariant and $\theta\circ i_G=V$;
with such a choice, if $(B,\beta)$ is any action and $\sigma\:V\rtimes_\beta G\variso C$ is a $\what\beta-\zeta$ equivariant isomorphism such that $\sigma\circ i_G=V$,
then there exists a conjugacy $\varphi\:(A,\alpha)\variso (B,\beta)$ such that $\sigma\circ(\varphi\rtimes G)=\theta$.
In fact, we can take $A$ to be the $C^*$-subalgebra of $M(C)$ consisting of all multipliers $m$ satisfying \emph{Landstad's conditions}
\begin{enumerate}
\item $\zeta_\gamma(m)=m$ for all $\gamma\in\what G$;

\item $mV(f),V(f)m\in C$ for all $f\in C_c(G)$;

\item $s\mapsto \ad V_s(m)$ is norm continuous from $G$ to $C$,
\end{enumerate}
and we can let $\alpha$ be the restriction of (the extension to $M(C)$ of) the inner action $\ad V$.
Then, letting $\iota\:A\to M(C)$ be the inclusion map,
the pair $(\iota,V)$ is a universal covariant homomorphism of the action $(A,\ad V)$.
We write $C^{\zeta,V}=A$, and call it the \emph{generalized fixed-point algebra} of the equivariant action $(C,\zeta,V)$.
Note that, given equivariant actions $(C,\zeta,V)$ and $(D,\epsilon,W)$ of $\what G$,
if $\Theta\:(C,\zeta)\variso (D,\epsilon)$ is a conjugacy such that $\Theta\circ V=W$,
then $\Theta(C^{\zeta,V})=D^{\epsilon,W}$.
In \cite[Corollary~2.6]{koqlandstad} we recorded a routine consequence of Landstad duality that we sometimes find useful (which in the following rendition we translate into the context of abelian groups):
if $(C,\zeta,V)$ is an equivariant action of $\what G$
and $\varphi\:A\to C^{\zeta,V}$ is an isomorphism,
then there exist an action $\alpha$ of $G$ on $A$
and a conjugacy
\[
\Theta\:(A\rtimes_\alpha G,\what\alpha)\variso (C,\zeta)
\]
such that
\[
\Theta\circ i_G=V
\midtext{and}
\Theta\circ i_A=\varphi.
\]

A \emph{cocycle} for an action $(A,\alpha)$ is a strictly continuous unitary map $u\:G\to M(A)$ such that $u_{st}=u_s\alpha_s(u_t)$ for all $s,t\in G$.
In this case $\ad u\circ\alpha$ is also an action of $G$ on $A$,
which is called \emph{exterior equivalent} to $\alpha$.
In particular, a strictly continuous unitary homomorphism is a cocycle for the trivial action $\iota$,
and an inner action is exterior equivalent to $\iota$.
Two actions $(A,\alpha)$ and $(B,\beta)$ are \emph{outer conjugate} if $\beta$ is conjugate to an action on $A$ that is exterior equivalent to $\alpha$.
If $\alpha$ and $\beta$ are outer conjugate, then the dual actions $\what\alpha$ and $\what\beta$ are conjugate.
Also, outer conjugacy is stronger than Morita equivalence
(strictly so, because $(A,\alpha)$ and $(B,\beta)$ outer conjugate implies $A\simeq B$).

\emph{Pedersen's theorem} \cite[Theorem~35]{pedersenexterior}
(stated more precisely in \cite[Theorem~0.10]{raeros},
and generalized to nonabelian groups in \cite[Theorem~3.1]{koqlandstad})
says that two actions $\alpha$ and $\beta$ of $G$ on $A$ are exterior equivalent if and only if there is a conjugacy $\Theta\:(A\rtimes_\alpha G,\what\alpha)\variso (A\rtimes_\beta G,\what\beta)$ such that $\Theta\circ i_A^\alpha=i_A^\beta$.
We take this opportunity to record an alternative version, which seems not to be explicitly recorded in the literature,
but which is an immediate consequence of \cite[proof of Theorem~5.9]{koqlandstad}:

\begin{thm}[Pedersen's theorem]\label{ped outer}
Two actions $(A,\alpha)$ and $(B,\beta)$ of $G$ are outer conjugate if and only if there is an isomorphism
\[
\Theta\colon (A\rtimes_\alpha G,\what\alpha)\variso (B\rtimes_\beta G,\what\beta)
\]
such that
\begin{equation}\label{same gfpa}
\Theta(i_A(A))=i_B(B).
\end{equation}
\end{thm}

\begin{proof}
The forward direction is an obvious corollary of Pedersen's theorem,
so suppose that $\Theta$ is an isomorphism satisfying \eqref{same gfpa}.
Let $V=\Theta\circ i_G^\alpha$.
Then $(B\rtimes_\beta G,\what\beta,V)$ is an equivariant action of $\what G$,
with
\begin{align*}
(B\rtimes_\beta G)^{\what\beta,V}
&=\Theta\bigl((A\rtimes_\alpha G)^{\what\alpha,i_G^\alpha}\bigr)
=\Theta(i_A(A))=i_B^\beta(B),
\end{align*}
and $i_B^\beta\:B\to i_B^\beta(B)$ is an isomorphism,
so by Landstad duality
there are an action $\zeta$ of $G$ on $B$
and a conjugacy
\[
\Psi\:(B\rtimes_\zeta G,\what\zeta)\variso (B\rtimes_\beta G,\what\beta)
\]
such that
\[
\Psi\circ i_G^\zeta=V
\midtext{and}
\Psi\circ i_B^\zeta=i_B^\beta.
\]
Thus by Pedersen's theorem the actions $\beta$ and $\zeta$ are exterior equivalent.

On the other hand, we have a conjugacy
\[
\Psi\inv\circ\Theta\:(A\rtimes_\alpha G,\what\alpha)\variso (B\rtimes_\zeta G,\what\zeta),
\]
taking $i_G^\alpha$ to $i_G^\zeta$,
so again by Landstad duality the actions $\alpha$ and $\zeta$ are conjugate.
Therefore $\alpha$ and $\beta$ are outer conjugate.
\end{proof}

\begin{defn}
We will refer to \eqref{same gfpa} as \emph{Pedersen's condition}.
\end{defn}

Thus, Pedersen's theorem says that two actions of $G$ are outer conjugate if and only if there is a conjugacy between the dual actions satisfying Pedersen's condition.

It is useful to compare Pedersen's theorem (\thmref{ped outer}) to Landstad duality:
it follows from the latter that two actions $(A,\alpha)$ and $(B,\beta)$ of $G$ are conjugate if and only if there 
is
an isomorphism
\[
\Theta\colon (A\rtimes_\alpha G,\what\alpha)\variso (B\rtimes_\beta G,\what\beta)
\]
such that
\[
\Theta\circ i_G^\alpha=i_G^\beta.
\]

\section{The Pedersen rigidity problem}\label{problem}

The alternative form of Pedersen's theorem (\thmref{ped outer})
leads to
\theoremstyle{plain}
\newtheorem*{gfpa*}{The Pedersen rigidity problem}
\begin{gfpa*}
Does there exist an example of two non-outer-conjugate actions $\alpha$ and $\beta$ of $G$ such that the dual actions $\what\alpha$ and $\what\beta$ are conjugate?
Equivalently,
is Pedersen's condition $\Theta(i_A(A))=i_B(B)$ in \thmref{ped outer} redundant?
\end{gfpa*}

We have only found a few references to this question in the literature;
for example,
Buss and Echterhoff say in \cite[Remark~3.13 (e)]{BusEch2} that it ``is not clear to us'' whether ``there might exist two different structures''
giving
``different generalized fixed-point algebras''.

First, we remark the following: suppose that we have an isomorphism $\varphi\colon A\variso B$ and two actions $\beta$ and $\gamma$ on $B$ that are \emph{not} exterior equivalent. Then define an action $\alpha$ on $A$ by setting
\[
\alpha_g(a)=\varphi^{-1}(\beta_g(\varphi(a))).
\]
Then $(A,\alpha)$ and $(B,\gamma)$ are \emph{not} outer conjugate. In other words, there are two ways to produce non-outer-conjugate actions; either via non-isomorphic $C^*$-algebras, or via non-exterior-equivalent actions.

In the following sections we present various no-go theorems,
giving general conditions
under which Pedersen's condition is redundant.
Our no-go theorems seem to come in two flavors, which we characterize via the following definitions:

\begin{defn}
An action $(C,\zeta)$ of $\what G$ is \emph{strongly fixed-point rigid} if it has a unique generalized fixed-point algebra,
i.e., for any two $\what G$-equivariant strictly continuous unitary homomorphisms $V,W\:G\to M(C)$ we have
\[
C^{\zeta,V}=C^{\zeta,W}.
\]
An action $(A,\alpha)$ of $G$ is \emph{strongly Pedersen rigid} if its dual action is strongly fixed-point rigid.
\end{defn}

\begin{defn}
An action $(C,\zeta)$ of $\what G$ is \emph{fixed-point rigid} if
the automorphism group of $(C,\zeta)$ acts transitively on the set of generalized fixed-point algebras,
i.e., for any two $\what G$-equivariant strictly continuous unitary homomorphisms $V,W\:G\to M(C)$
there is an automorphism $\Theta$ of $(C,\zeta)$ such that
\[
\Theta(C^{\zeta,V})=C^{\zeta,W}.
\]
An action $(A,\alpha)$ of $G$ is \emph{Pedersen rigid} if its dual action is fixed-point rigid.
\end{defn}

\begin{lem}\label{fixed point rigid lem}
Let $(C,\zeta,V)$ be an equivariant action of $\what G$, and consider the following properties:
\begin{enumerate}
\item $\zeta$ is strongly fixed-point rigid.

\item For every equivariant action $(D,\epsilon,W)$ of $\what G$,
if
$\Theta\:(C,\zeta)\variso (D,\epsilon)$
is a conjugacy then
$\Theta(C^{\zeta,V})=D^{\epsilon,W}$.

\item $\zeta$ is fixed-point rigid.

\item For every equivariant action $(D,\epsilon,W)$ of $\what G$,
if
$\zeta$ and $\epsilon$ are conjugate then
there is a conjugacy
$\Theta\:(C,\zeta)\variso (D,\epsilon)$
such that
$\Theta(C^{\zeta,V})=D^{\epsilon,W}$.
\end{enumerate}
Then 
$(1)\iff (2) \then (3) \iff (4)$.
\end{lem}

\begin{proof}
Assume (1), and let $\Theta\:(C,\zeta)\variso (D,\epsilon)$ be a conjugacy.
Then $\Theta\inv\circ W$ is a $\what G$-equivariant strictly continuous unitary homomorphism, so
\begin{align*}
C^{\zeta,V}
&=C^{\zeta,\Theta\inv\circ W}\righttext{(since $(C,\zeta)$ is strongly fixed-point rigid)}
\\&=\Theta\inv\bigl(D^{\epsilon,W}),
\end{align*}
and we have shown (2).
Conversely, $(2)\then (1)$ follows by taking $(D,\epsilon)=(C,\zeta)$ and $\Theta=\id_C$.

$(1)\then (3)$ is trivial, and
$(3)\iff (4)$ is similar to $(1)\iff (2)$.
\end{proof}

\begin{cor}\label{rigid vs outer rigid}
Let $(A,\alpha)$ be an action of $G$, and consider the following properties:
\begin{enumerate}
\item $\alpha$ is strongly Pedersen rigid.

\item For every action $(B,\beta)$ of $G$,
if
$\Theta\:(A\rtimes_\alpha G,\what\alpha)\variso (B\rtimes_\beta G,\what\beta)$
is a conjugacy then
$\Theta(i_A(A))=i_B(B)$.

\item $\alpha$ is Pedersen rigid.

\item If $(B,\beta)$ is any action of $G$, then
$\alpha$ and $\beta$ are outer conjugate if and only if
$\what\alpha$ and $\what\beta$ are conjugate.
\end{enumerate}
Then 
$(1)\iff (2) \then (3) \iff (4)$.
\end{cor}

\begin{proof}
Recall that $i_A(A)=(A\rtimes_\alpha G)^{\what\alpha,i_G^\alpha}$,
and similarly for $i_B(B)$.
By Pedersen's theorem,
condition (4) is equivalent to the following:
for any action $(B,\beta)$ of $G$,
if $\what\alpha$ and $\what\beta$ are conjugate then
there is a conjugacy
$\Theta\:(A\rtimes_\alpha G,\what\alpha)\variso (B\rtimes_\beta G,\what\beta)$
such that
$\Theta(i_A(A))=i_B(B)$.
Now
apply \lemref{fixed point rigid lem}
with $(C,\zeta,V)=(A\rtimes_\alpha G,\what\alpha,i_G^\alpha)$,
keeping in mind that
if $(D,\epsilon,W)$ is any equivariant action of $\what G$,
then by Landstad duality there is an action $(B,\beta)$ of $G$ and a conjugacy $\Psi\:(D,\epsilon)\variso (B\rtimes_\beta G,\what\beta)$ such that $\Psi\circ W=i_G^\beta$.
\end{proof}

Thus, the Pedersen rigidity problem is equivalent to the following:
\begin{gfpa*}[alternative formulation]
Is every action of $G$ Pedersen rigid?
\end{gfpa*}

But now with the stronger type of rigidity, we can ask for more:
\theoremstyle{plain}
\newtheorem*{sgfpa*}{The strong Pedersen rigidity problem}
\begin{sgfpa*}
Is every action of $G$ strongly Pedersen rigid?
\end{sgfpa*}
In fact, our no-go theorems seem to hint that this stronger rigidity might hold.

In the next section, we will discuss the rigidity problem when restricting to certain types of groups and actions, and therefore we introduce the following:

\begin{defn}
We say that 
a class $\CC$ of actions of $G$
is \emph{Pedersen rigid}
if two actions in $\CC$ are outer conjugate if and only if their dual actions are conjugate.
\end{defn}

\begin{rem}
Suppose we are given an action $(C,\zeta)$ of $\what G$.
Let us write $H(G,C,\zeta)$ for the set of all $\what G$-equivariant strictly continuous unitary homomorphisms $V\:G\to M(C)$.
Note that the group $\aut(C,\zeta)$ acts on this set by composition:
\[
(\theta,V)\mapsto \theta\circ V\:\aut(C,\zeta)\times H(G,C,\zeta)\to H(G,C,\zeta).
\]
Now let us write $GFPA(C,\zeta)$ for the set of all generalized fixed-point algebras of the action $(C,\zeta)$.
Then we have a surjection
\[
V\mapsto C^{\zeta,V}\:H(G,C,\zeta)\to GFPA(C,\zeta),
\]
and the action of $\aut(C,\zeta)$ descends to an action on $GFPA(C,\zeta)$.
By definition, the action $(C,\zeta)$ is fixed-point rigid if and only if this action on $GFPA(C,\zeta)$ is transitive.
This leads us to consider a stronger property: can the action of $\aut(C,\zeta)$ on $H(G,C,\zeta)$ be transitive?
We will explain here that the answer is generally negative.

For every $V\in H(G,C,\zeta)$,
by Landstad duality there are
an action $(A,\alpha)$ of $G$
and a conjugacy
\[
\Theta\:(A\rtimes_\alpha G,\what\alpha)\variso (C,\zeta)
\]
such that
\[
\Theta\circ i_G^\alpha=V.
\]
In other words, the elements $V\in H(G,C,\zeta)$ correspond to actions $(A,\alpha)$ of $G$ for which
$(A\rtimes_\alpha G,\what\alpha,i_G^\alpha)=(C,\zeta,V)$.
If
we take another $W\in H(G,C,\zeta)$ and
another action $(B,\beta)$ of $G$ with
$(B\rtimes_\beta G,\what\beta,i_G^\beta)=(C,\zeta,W)$,
then it follows from Landstad duality that
$(A,\alpha)\simeq (B,\beta)$ 
if and only if
there exists $\theta\in \aut(C,\zeta)$ such that $\theta\circ V=W$.
We can arrange for $(B,\beta)$ to be outer conjugate,
but \emph{not conjugate},
to $(A,\alpha)$, and
then there is no automorphism of $(C,\zeta)$ taking $V$ to $W$.

As we mentioned above, our no-go theorems hint at the possibility
that every action of $G$ is actually strongly Pedersen rigid.
If so, then the set $GFPA(C,\zeta)$ would be a singleton, but again we could easily have the action of $\aut(C,\zeta)$ on $H(G,C,\zeta)$ be nontransitive.
\end{rem}

\begin{rem}
The above definition of fixed-point rigidity (as well as strong fixed-point rigidity) can be phrased in terms of actions of an arbitrary locally compact group $\Gamma$ instead of $\what G$,
but then the homomorphism $V\:G\to M(C)$ would have to be replaced by a nondegenerate homomorphism from $C_0(\Gamma)$ to $M(C)$ that is equivariant for $\zeta$ and the action of $\Gamma$ on $C_0(\Gamma)$ given by right translation.

Moreover, \thmref{ped outer} also holds for non-abelian groups, again by applying \cite[proof of Theorem~5.9]{koqlandstad}. Hence, the Pedersen rigidity problem can be formulated for arbitrary locally compact groups as well.
\end{rem}

\begin{rem}
It is interesting to note that the rigidity theory for $C^*$-dynamical systems discussed in this paper bears resemblance to recent works involving diagonal-preserving isomorphisms between graph $C^*$-algebras, see \cite[Theorem~1.5]{Matsumoto} and \cite[Theorem~4.1]{Carlsen-Rout}.
\end{rem}

\section{No-go theorems}

\subsection{Discrete groups}

If $G$ is discrete, then $\what G$ is compact, so the dual action has a genuine fixed-point algebra, and hence all generalized fixed-point algebras 
coincide. Consequently, we get the following result (which we
have not found in the literature):

\begin{prop}[\nogothm]\label{nogo discrete}
If $G$ is discrete,
then 
every action of $G$ is strongly Pedersen rigid.
\end{prop}

Thus, by Pedersen's theorem, if
$G$ is discrete,
then
two actions of $G$ 
are outer conjugate if and only if
the dual actions are conjugate.

\subsection{Stable $C^*$-algebras}

If $(A\rtimes_\alpha G,\what{\alpha}) \simeq (B\rtimes_\beta G,\what\beta)$,
then by Takai duality
the actions $\alpha$ and $\beta$ must at least be Morita equivalent.

Moreover, \cite[Section~8 Proposition]{com} says that if $A$ and $B$ are stable and have strictly positive elements (which is satisfied if they are separable, for example) then $\alpha$ and $\beta$ are Morita equivalent if and only if they are outer conjugate.
This leads to

\begin{prop}[\nogothm]\label{nogo stable}
The class of actions of $G$ on stable $C^*$-algebras possessing strictly positive elements is Pedersen rigid.
\end{prop}

Thus
(assuming, for example, that we restrict our attention to separable $C^*$-algebras)
we will not find any examples of multiple generalized fixed-point algebras unless at least one of $A$ and $B$ is nonstable. This indicates that the phenomenon of multiple generalized fixed-point algebras is delicate in some sense, since it would not be possible, for example, if we replace the original actions by their double duals, since double crossed products are always stable.

\subsection{Commutative $C^*$-algebras}

For commutative algebras, we have a stronger version of \propref{nogo stable}.

\begin{prop}[\nogothm]\label{nogo commutative}
If $A$ and $B$ are commutative,
then actions $(A,\alpha)$ and $(B,\beta)$ of $G$ 
are conjugate if and only if
the dual actions are conjugate.
In particular, 
the class of actions of $G$ on commutative $C^*$-algebras is Pedersen rigid.
\end{prop}

\begin{proof}
It suffices to show that if the actions $\alpha$ and $\beta$ are Morita equivalent then they are conjugate.
Suppose we have an $\alpha-\beta$ equivariant $A-B$ imprimitivity bimodule.
Then the associated Rieffel homeomorphism
$\what B\simeq \what A$
\cite[Corollary~6.27]{rie:induced}
is $G$-equivariant,
and this gives an $\alpha-\beta$ equivariant isomorphism $A\simeq B$.
\end{proof}

\begin{rem}
If we have an action $(A,\alpha)$ and we know that $A$ is commutative, then the dual action $\what\alpha$ contains all the information about $\alpha$ (up to conjugacy), so weaker forms of equivalence are of interest. In a recent paper by Li \cite{LiCOER}, a notion of continuous orbit equivalence for topological dynamical systems is discussed:

Let $X$ and $Y$ be locally compact Hausdorff spaces on which $G$ acts. Then $(X,G)$ and $(Y,G)$ are said to be \emph{continuously orbit equivalent} if there exists a homeomorphism $\varphi\colon X\to Y$ and continuous maps $a\colon G\times X\to G$ and $b\colon G\times Y\to G$ such that $\varphi(g\cdot x)=a(g,x)\cdot\varphi(x)$ and $\varphi^{-1}(g\cdot y)=b(g,y)\cdot\varphi^{-1}(y)$. This is clearly weaker than conjugacy (which is obtained by setting $a(g,x)=b(g,y)=g$).

Moreover, if $G$ is discrete and the actions are topologically free, then \cite[Theorem~1.2]{LiCOER} says that $(X,G)$ and $(Y,G)$ are continuously orbit equivalent if and only if there exists an isomorphism $\Phi\colon C_0(X)\rtimes G\to C_0(Y)\rtimes G$ satisfying Pedersen's condition.
\end{rem}

The no-go theorems Propositions~\ref{nogo stable} and \ref{nogo commutative} have consequences for equivariant actions:

\begin{cor}\label{stable commutative}
Let $(C,\zeta)$ be an action of $\what G$,
and let $V,W\:G\to M(C)$ be $\what G$-equivariant strictly continuous unitary homomorphisms.
If both $C^{\zeta,V}$ and $C^{\zeta,W}$ are stable and have strictly positive elements,
or both are commutative,
then there is a $\zeta$-equivariant automorphism of $C$ that takes $C^{\zeta,V}$ to $C^{\zeta,W}$.
\end{cor}

\begin{proof}
By Landstad duality, there exist actions $(A,\alpha)$ and $(B,\beta)$ of $G$
and
conjugacies
\[
\xymatrix{
(A\rtimes_\alpha G,\what\alpha) \ar[r]^-\theta_-\simeq
&(C,\zeta)
&(B\rtimes_\beta G,\what\beta) \ar[l]_-\psi^-\simeq
}
\]
such that
\[
\theta\circ i_G^\alpha=V,
\quad
\theta(i_A(A))=C^{\zeta,V},
\quad
\psi\circ i_G^\beta=W,
\quad
\psi(i_B(B))=C^{\zeta,W}.
\]
Then
we have a conjugacy
\[
\psi\inv\circ\theta\:
(A\rtimes_\alpha G,\what\alpha)
\variso
(B\rtimes_\beta G,\what\beta),
\]
and hence,
by Propositions~\ref{nogo stable} and \ref{nogo commutative},
$(A,\alpha)$ and $(B,\beta)$ are outer conjugate.
Thus there exists a conjugacy
\[
\sigma\:
(A\rtimes_\alpha G,\what\alpha)
\variso
(B\rtimes_\beta G,\what\beta)
\]
(possibly different from $\psi\inv\circ\theta$)
such that
\[
\sigma(i_A(A))=i_B(B).
\]
Then $\psi\circ\sigma\circ\theta\inv$
is a $\zeta$-equivariant automorphism of $C$
such that
\begin{align*}
\psi\circ\sigma\circ\theta\inv(C^{\zeta,V})
&=\psi\circ\sigma(i_A(A))
=\psi(i_B(B))
=C^{\zeta,W}.
\qedhere
\end{align*}
\end{proof}

\subsection{Inner actions}

Any inner action $\alpha=\ad u$ determined by a unitary homomorphism $u\:G\to M(A)$
is exterior equivalent to the trivial action $\iota$.
Thus we can take the dual actions to be the same:
\[
(A\rtimes_\alpha G,\what\alpha)=(A\rtimes_\iota G,\what\iota).
\]
The homomorphism $i_G^\iota\:G\to M(A\rtimes_\iota G)$ maps into the center,
and hence it commutes with $i_G^\alpha\:G\to M(A\rtimes_\iota G)$.
Therefore, by
\cite[Lemma~1.6]{QR95}
(see also \cite[Proposition~3.12]{BusEch2} for a slightly more general result)
the generalized fixed-point algebras coincide:
\[
(A\rtimes_\alpha G)^{\what\alpha,i_G^\alpha}=(A\rtimes_\iota G)^{\what\iota,i_G^\iota}.
\]
By transitivity, if we are given another action $(B,\beta)$ such that
\[
(A\rtimes_\alpha G,\what\alpha)\simeq (B\rtimes_\beta G,\what\beta),
\]
then with a bit more work we will have:

\begin{prop}[\nogothm]\label{nogo inner}
Every inner action is strongly Pedersen rigid.
\end{prop}

\begin{proof}
We will apply \lemref{rigid vs outer rigid} $(1)\iff (2)$.
Let $(A,\alpha)$ and $(B,\beta)$ be actions of $G$
such that $\alpha$ is inner and $\what\alpha$ is conjugate to $\what\beta$.
Up to isomorphism, we can take
\[
(A\rtimes_\iota G,\what{\iota})=(A\rtimes_\alpha G,\what{\alpha})
=(B\rtimes_\beta G,\what\beta).
\]
Then we have three equivariant unitary homomorphisms
\[
i_G^\iota,i_G^\alpha,i_G^\beta\:G\to M(A\rtimes_\iota G).
\]
The first one maps into the center, and hence commutes with the other two.
Thus
\[
(A\rtimes_\alpha G)^{\what\alpha,i_G^\alpha}
=(A\rtimes_\iota G)^{\what\iota,i_G^\iota}
=(B\rtimes_\beta G)^{\what\beta,i_G^\beta}.
\]
Now we revert back to the original situation,
without the ``up to isomorphism'' reduction,
and we find that any isomorphism $\Theta\:(A\rtimes_\alpha G,\what\alpha)\variso (B\rtimes_\beta G,\what\beta)$ will preserve the generalized fixed-point algebras:
\[
\Theta\bigl((A\rtimes_\alpha G)^{\what\alpha,i_G^\alpha}\bigr)
=(B\rtimes_\beta G)^{\what\beta,i_G^\beta}.
\qedhere
\]
\end{proof}

\begin{rem}
\propref{nogo inner} could be generalized by using the full force of \cite[Proposition~3.12]{BusEch2}:
given two actions $(A,\alpha)$ and $(B,\beta)$,
if there is an isomorphism $\Theta\:(A\rtimes_\alpha G,\what\alpha)\variso (B\rtimes_\beta G,\what\beta)$ such that the two sets of products
\[
\Theta(i_G^\alpha(C^*(G)))i_G^\beta(C^*(G))\midtext{and}i_G^\beta(C^*(G))\Theta(i_G^\alpha(C^*(G)))
\]
coincide, then $\alpha$ and $\beta$ are outer conjugate.
\end{rem}

\subsection{Ergodic actions of compact groups}

Suppose that $G$ is compact,
$A$ is a unital $C^*$-algebra,
and $\alpha$ is an action of $G$ on $A$ that is ergodic,
that is, the fixed-point algebra is $A^\alpha=\C 1$.
We refer to
\cite{OlePedTak}
for the following facts concerning ergodic actions of compact abelian groups.
For each $\gamma\in\what G$ let $A_\gamma$ be the associated spectral subspace
\[
A_\gamma=\{a\in A:\alpha_s(a)=\gamma(s)a\text{ for all }s\in G\}.
\]
We impose the further hypothesis that
the action $\alpha$ is faithful (i.e., $\alpha_s=\alpha_t$ implies $s=t$),
and consequently
\[
A_\gamma\ne \{0\}\righttext{for all}\gamma\in \what G.
\]
Then for all $\gamma\in\what G$ there is a unitary $u_\gamma\in A$ such that
\[
A_\gamma=\C u_\gamma.
\]
Moreover, there is a 2-cocycle $\omega\:\what G\times \what G\to\T$ such that
\begin{equation}\label{omega-rep}
u_\gamma u_\chi=\omega(\gamma,\chi)u_{\gamma\chi};
\end{equation}
unitary-valued maps $u$ of $G$ satisfying \eqref{omega-rep} are called
\emph{$\omega$-representations}
(see \cite[Appendix~D.3]{danacrossed}, for example),
and we say that $\omega$ is \emph{compatible} with $(A,\alpha)$.
Recall that a map
$\omega\:\what G\times \what G\to\T$
is called a \emph{2-cocycle}
if
\[
\omega(\gamma,\chi)\omega(\gamma\chi,\sigma)
=\omega(\gamma,\chi\sigma)\omega(\chi,\sigma)
\righttext{for all}\gamma,\chi,\sigma\in \what G.
\]
We will assume that all our 2-cocycles are \emph{normalized}, namely
\[
\omega(1,\gamma)=\omega(\gamma,1)=1\righttext{for all}\gamma\in \what G.
\]
The set $Z^2(\what G,\T)$ of all 2-cocycles
is an abelian group under pointwise multiplication.
A cocycle $\omega\in Z^2(\what G,\T)$ is called a \emph{coboundary} if there is a map $z\:\what G\to\T$ such that
\[
\omega(\gamma,\chi)=z(\gamma)z(\chi)\bar{z(\gamma\chi)}
\righttext{for all}\gamma,\chi\in \what G.
\]
The set $B^2(\what G,\T)$ of all coboundaries is a subgroup of $Z^2(\what G,\T)$,
and the quotient group
$H^2(\what G,\T):=Z^2(\what G,\T)/B^2(\what G,\T)$
is the \emph{second cohomology group} of $\what G$ with values in $\T$.
Two 2-cocycles $\omega$ and $\zeta$ of $\what G$ are called \emph{cohomologous} if
$\omega B^2(\what G,\T)=\zeta B^2(\what G,\T)$,
i.e.,
there is a map $z\:\what G\to\T$ such that
\[
\omega(\gamma,\chi)=z(\gamma)z(\chi)\bar{z(\gamma\chi)}\zeta(\gamma,\chi)\righttext{for all}\gamma,\chi\in \what G.
\]
Faithful ergodic actions of $G$ are classified by $H^2(\what G,\T)$ in the following sense:
two faithful ergodic actions $(A,\alpha)$ and $(B,\beta)$ of $G$,
with compatible 2-cocycles $\omega$ and $\zeta$,
are conjugate if and only if $\omega$ and $\zeta$ are cohomologous.

Let 
$\alpha$ be a faithful ergodic action of $G$ on $A$,
let $\omega$ be an $\alpha$-compatible cocycle,
and let
$\lambda^\omega$ be the associated left regular $\omega$-representation of $\what G$ on $\ell^2(\what G)$,
given by
\[
(\lambda^\omega_\gamma \xi)(\chi)=\omega(\gamma,\bar\gamma\chi)\xi(\bar\gamma \chi)\righttext{for}\gamma,\chi\in\what G,\xi\in \ell^2(\what G).
\]
Then there is an isomorphism
\[
(A\rtimes_\alpha G,\what\alpha)\simeq (\KK(\ell^2(\what G)),\ad\lambda^\omega)
\]
(see \cite[Remark~6.7]{OlePedTak}).

\begin{prop}[\nogothm]\label{nogo ergodic}
If $G$ is compact, then
faithful ergodic actions $(A,\alpha)$ and $(B,\beta)$ of $G$
are conjugate if and only if the dual actions are conjugate.
In particular,
the class of faithful ergodic actions of $G$ is Pedersen rigid.
\end{prop}

\begin{proof}
It suffices to prove that if $(A,\alpha)$ and $(B,\beta)$ are faithful ergodic actions of $G$
such that
\[
(A\rtimes_\alpha G,\what\alpha)\simeq (B\rtimes_\beta G,\what\beta),
\]
then $(A,\alpha)\simeq (B,\beta)$.
Let $\omega$ and $\zeta$ be cocycles of $\what G$ compatible with $\alpha$ and $\beta$, respectively.
By \lemref{cocycle actions conjugate} below,
the cocycles $\omega$ and $\zeta$ are cohomologous,
and hence the actions $(A,\alpha)$ and $(B,\beta)$ are conjugate.
\end{proof}

We thank Magnus Landstad for conversations that ultimately led to the above no-go theorem.

In the above proof we referred to the following lemma, which we state in abstract form, for a possibly nonabelian discrete group.
The result is contained in
\cite[Proposition~D.27]{danacrossed}, but we give a short proof for convenience.

\begin{lem}\label{cocycle actions conjugate}
Let $\omega_1$ and $\omega_2$ be 2-cocycles of a discrete group $\Gamma$,
and let $\pi_i$ be an $\omega_i$-representation for $i=1,2$.
Suppose that the actions $(\KK(\ell^2(\Gamma)),\ad\pi_1)$ and $(\KK(\ell^2(\Gamma)),\ad\pi_2)$ are conjugate.
Then $\omega_1$ and $\omega_2$ are cohomologous.
\end{lem}

\begin{proof}
Since all automorphisms of $\KK(\ell^2(\Gamma))$ are inner, we can choose a unitary $u\in\UU(\ell^2(\Gamma))$ such that
\[
\pi_2(g)uxu^*\pi_2(g)^*=u\pi_1(g)x\pi_1(g)^*u^*
\]
for all $g\in\Gamma$ and all $x\in\KK(\ell^2(\Gamma))$.
This means that 
\[
\pi_1(g)^*u^*\pi_2(g)ux=x\pi_1(g)^*u^*\pi_2(g)u,
\]
for all $g\in\Gamma$ and all $x\in\KK(\ell^2(\Gamma))$, hence there are scalars $z(g)\in\C$ such that
\[
\pi_1(g)^*u^*\pi_2(g)u=z(g)1
\]
for all $g\in\Gamma$, i.e., such that
\[
\pi_2(g)=z(g)u\pi_1(g)u^*
\]
for all $g\in\Gamma$.
Define the $\omega_1$-projective representation $\pi_1'$ by $\pi_1'(g)=u\pi_1(g)u^*$ for all $g\in\Gamma$.
Then $\pi_2(g)=z(g)\pi_1'(g)$ for all $g\in\Gamma$, and
\begin{align*}
\omega_2(g,h)z(gh)\pi_1'(gh)
&=\omega_2(g,h)\pi_2(gh)
\\&=\pi_2(g)\pi_2(h)
\\&=z(g)z(h)\pi_1'(g)\pi_1'(h)
\\&=z(g)z(h)\omega_1(g,h)\pi_1'(gh).
\end{align*}
for all $g,h\in\Gamma$. Hence,
\[
\omega_2(g,h)=z(g)z(h)\overline{z(gh)}\omega_1(g,h)
\]
for all $g,h\in\Gamma$, so $\omega_1$ and $\omega_2$ are cohomologous.
\end{proof}

\subsection{Local rigidity}

\begin{prop}[\nogothm]\label{nogo span}
Let $(A,\alpha)$ be an action,
and let $\II$ be a family of $\alpha$-invariant ideals of $A$
such that $A=\clspn \II$.
If for each $I\in\II$ the restricted action $\alpha^I$ is strongly Pedersen rigid,
then $\alpha$ is strongly Pedersen rigid.
\end{prop}

\begin{proof}
Let $(B,\beta)$ be an action, and suppose that we have a conjugacy
\[
\Theta\:(A\rtimes_\alpha G,\what\alpha)\variso (B\rtimes_\beta G,\what\beta).
\]
By \lemref{rigid vs outer rigid}, it suffices to show that $\Theta(i_A(A))=i_B(B)$.
Let $I\in\II$, and let $\alpha^I$ denote the restricted action on $I$.
Then $I\rtimes_{\alpha^I} G$ is an $\what\alpha$-invariant ideal of $A\rtimes_\alpha G$,
so $K:=\Theta(I\rtimes G)$ is a $\what\beta$-invariant ideal of $B\rtimes_\beta G$.
Then by \cite[Theorem~3.4]{gl}
(since $G$ is amenable)
there is a $\beta$-invariant ideal $J_I$ of $B$ such that
\[
(J_I\rtimes G,\what{\beta^{J_I}})=(K,\what\beta^K).
\]
Thus, by \lemref{rigid vs outer rigid} again,
\[
\Theta(i_I(I))=i_{J_I}(J_I).
\]
But
$i_I$ is the restriction of $i_A$ to the ideal $I$,
and similarly
$i_{J_I}=i_B|_{J_I}$,
so we have
\[
\Theta(i_A(I))=i_B(J_I).
\]

Now, we have a family of $\beta$-invariant ideals $\{J_I\}_{I\in\II}$ of $B$,
and
\begin{align*}
B\rtimes_\beta G
&=\Theta(A\rtimes_\alpha G)
\\&=\Theta\left(\clspn_{I\in\II} I\rtimes G\right)
\\&=\clspn_{I\in\II}\Theta(I\rtimes G)
\\&\hspace{.5in}\text{(since $\Theta$ is an isomorphism between $C^*$-algebras)}
\\&=\clspn_{I\in\II}J_I\rtimes G,
\end{align*}
so by \lemref{ideal span} below we have $B=\clspn_{I\in\II}J_I$.
Therefore,
\begin{align*}
\Theta(i_A(A))
&=\Theta\left(i_A\Bigl(\clspn\{I:I\in\II\}\Bigr)\right)
\\&=\clspn\left\{\Theta(i_A(I)):I\in\II\right\}
\\&=\clspn\left\{i_B(J_I):I\in\II\right\}
\\&=i_B\left(\clspn\{J_I:I\in\II\}\right)
\\&=i_B(B).
\qedhere
\end{align*}
\end{proof}

In the above proof we referred to the following general lemma, which is presumably folklore:

\begin{lem}\label{ideal span}
Let $(B,\beta)$ be an action, and let $\{J_I\}_{I\in\II}$ be a family of $\beta$-invariant ideals of $B$ such that
\[
B\rtimes_\beta G=\clspn_{I\in\II}J_I\rtimes G.
\]
Then $B=\clspn_{I\in\II}J_I$.
\end{lem}

\begin{proof}
Put
\[
J=\clspn_{I\in\II}J_I.
\]
Then $J$ is a $\beta$-invariant ideal of $B$,
and
\begin{align*}
J\rtimes G
&=\clspn_{I\in\II}J_I\rtimes G\righttext{(by \cite[Proposition~9]{gre:local})}
\\&=B\rtimes_\beta G,
\end{align*}
and so we must have $J=B$,
by
\cite[Proposition~11]{gre:local}.
\end{proof}

We will apply \propref{nogo span} to locally unitary actions.
Recall from \cite[Section~7.5]{tfb} that an action $(A,\alpha)$ is called \emph{locally unitary} if $A$ is continuous-trace and for each $\pi\in\what A$ there exist a compact neighborhood $F$ of $\pi$ and a strictly continuous unitary homomorphism $u\:G\to M(A^F)$
(where $A^F$ is the corresponding quotient of $A$)
such that
$\alpha^F=\ad u$
(where $\alpha^F$ denotes the quotient action of $G$).
Since $\what A$ is locally compact Hausdorff
it is equivalent to require that
there exist an open set $N$ containing $\pi$
and a strictly continuous unitary homomorphism $u\:G\to M(A_N)$
(where $A_N$ is the corresponding ideal of $A$)
such that $\alpha=\ad u$ on $A_N$.
To see this, note that, given such a compact neighborhood $F$,
we could take $N$ to be the interior of $F$, and then,
identifying $A$ with the algebra of continuous sections of the associated $C^*$-bundle that vanish at infinity,
the corresponding ideal $A_N$ is the set of sections vanishing outside $N$,
while conversely given an open set $N$,
since $\what A$ is locally compact Hausdorff we can assume without loss of generality that $N$ has compact closure $F$,
and then the corresponding quotient $A^F$ is the set of continuous sections on the restricted $C^*$-bundle over $F$.
Then by \propref{nogo inner} the restricted actions on the $A_N$'s are strongly Pedersen rigid, and $A$ is the closed span of the $A_N$'s. Thus by \propref{nogo span} we have

\begin{cor}\label{nogo local}
Any locally unitary action on a continuous-trace $C^*$-algebra is strongly Pedersen rigid.
\end{cor}

\begin{q}
Is there 
an analogue of
\propref{nogo span} for quotients,
namely that if $\II$ is a family of $\alpha$-invariant ideals such that
$\bigcap \II=\{0\}$
and for each $I\in\II$ the quotient action on $A/I$ is strongly Pedersen rigid, then $\alpha$ is strongly Pedersen rigid?
This would allow us to strengthen \corref{nogo local} to pointwise unitary actions.
\end{q}

\begin{rem}
Our original motivation for studying the rigidity question for locally (or pointwise) unitary actions was \cite[Proposition~2.5]{phiraelocally}, which characterizes exterior equivalence for locally unitary actions in terms of isomorphism of the associated principal $\what G$-bundles
\[
\what{A\rtimes G}\to \what A
\]
(and \cite[Corollary~1.11]{olerae} is a version for pointwise unitary actions).
These characterizations of exterior equivalence made us wonder whether we could find examples of different generalized fixed-point algebras using these types of actions, but subsequent investigation led us instead to the above no-go theorem.
Thus, we have additional negative evidence for the existence of examples, 
alternatively, additional evidence in support of an affirmative answer to the Pedersen rigidity problem.
\end{rem}


\providecommand{\bysame}{\leavevmode\hbox to3em{\hrulefill}\thinspace}
\providecommand{\MR}{\relax\ifhmode\unskip\space\fi MR }
\providecommand{\MRhref}[2]{%
  \href{http://www.ams.org/mathscinet-getitem?mr=#1}{#2}
}
\providecommand{\href}[2]{#2}

\end{document}